\documentclass[12pt]{amsart}
\usepackage[T2A,T1]{fontenc}
\usepackage[utf8]{inputenc}		
\usepackage{ucs}
\usepackage[russian,english]{babel}	
\usepackage{amssymb}
\usepackage{ dsfont }
\usepackage{multirow}
\usepackage{amsthm}
\usepackage{tikz}
\usepackage{amsmath}
\usepackage{graphicx}
\usepackage{wrapfig}
\usepackage{caption}
\usepackage[left=2.5cm,right=2.5cm,top=2.4cm,bottom=2.4cm,bindingoffset=0cm]{geometry}
\usepackage{enumerate}
\usepackage{makecell}	
\newtheorem{state}{Statement}[section]
\newtheorem{T}{Theorem}
\newtheorem*{theorem}{Main Theorem}
\newtheorem*{remark}{Remark}
\newtheorem{cor}{Corollary}

\newtheorem{prob}{Problem}
\newtheorem{lemm}[state]{Lemma}
\newtheorem{lemma}[state]{Lemma}

\newtheorem{prop}[state]{Proposition}

\makeatletter
\newcommand*{\rom}[1]{\expandafter\@slowromancap\romannumeral #1@}
\makeatother

\makeatletter
\def\@seccntformat#1{\csname the#1\endcsname. }
\def\@biblabel#1{#1.}

\title[On characterization by Gruenberg--Kegel graph of finite simple groups]{On characterization by Gruenberg--Kegel graph of finite simple exceptional groups of Lie type}

\author{Natalia~V.~Maslova, Viktor~V.~Panshin, and Alexey~M.~Staroletov}
\address{Natalia Vladimirovna Maslova\newline Krasovskii Institute of Mathematics and Mechanics UB RAS,\newline
16, S. Kovalevskaja str.,
Yekaterinburg, 620108, Russia\newline
Ural Federal University,\newline
19, Mira str., Yekaterinburg, 620002, Russia}
\email{butterson@mail.ru}
\address{Viktor Vladimirovich Panshin \newline
Novosibirsk State University, \newline
1, Pirogova str., Novosibirsk, 630090, Russia\newline
Sobolev Institute of Mathematics, \newline
4, Acad. Koptyug ave., Novosibirsk, 630090, Russia}
\email{v.pansh1n@yandex.ru}
\address{Alexey Mikhailovich Staroletov \newline
Sobolev Institute of Mathematics, \newline
4, Acad. Koptyug ave., Novosibirsk, 630090, Russia}
\email{staroletov@math.nsc.ru}

\begin{document}

\maketitle
\hfill In memory of Irina Dmitrievna Suprunenko

\begin{abstract}
The Gruenberg--Kegel graph $\Gamma(G)$ of a finite group $G$ is the graph whose vertex set is the set of prime divisors of $|G|$ and in which two distinct vertices $r$ and $s$ are adjacent if and only if there exists an element of order $rs$ in $G$.

A finite group $G$ is called {\it almost recognizable} (by Gruenberg--Kegel graph) if
there is only finite number of pairwise non-isomorphic finite groups having
Gruenberg--Kegel graph as $G$. If $G$ is not almost recognizable, then it is called {\it unrecognizable}
(by Gruenberg--Kegel graph).

Recently P.~J.~Cameron and the first author have proved that if a finite group is almost recognizable, then the group is almost simple. Thus, the question of which almost simple groups (in particular, finite simple groups) are almost recognizable is of prime interest. We prove that every finite simple exceptional group of Lie type,
which is isomorphic to neither ${^2}B_2(2^{2n+1})$ with $n\geq1$ nor $G_2(3)$ and whose Gruenberg--Kegel graph has at least three connected components, is almost recognizable. Moreover, groups $ {^2}B_2(2^{2n+1})$, where $n\geq1$, and $G_2(3)$ are unrecognizable.
\end{abstract}

\medskip

Throughout the paper we consider only finite groups and simple graphs, and henceforth the term group means finite group and the term graph means simple graph, that is undirected graph without loops and multiple edges.

Let $G$ be a group. The {\it spectrum} $\omega(G)$ is the set of all element orders of $G$. The {\it prime spectrum} $\pi(G)$ is the set of all primes belonging to $\omega(G)$. A graph $\Gamma(G)$ whose vertex set is $\pi(G)$ and in which two distinct vertices~$r$ and~$s$ are adjacent if and only if $rs \in \omega(G)$ is called the {\it Gruenberg--Kegel graph} or the {\it prime graph} of~$G$.

\smallskip

\noindent We say that the group $G$ is
\begin{itemize}
\item{\it recognizable} (by Gruenberg--Kegel graph) if for every group $H$ the equality $\Gamma(H)=\Gamma(G)$ implies that $G\cong H${\rm;}
\item  {\it $k$-recognizable} (by Gruenberg--Kegel graph), where $k$ is a positive integer, if there are exactly $k$ pairwise non-isomorphic groups having the same Gruenberg--Kegel graph as~$G${\rm;}
\item  {\it almost recognizable} (by Gruenberg--Kegel graph) if it is $k$-recognizable by Gruenberg--Kegel graph for a positive integer $k${\rm;}
\item {\it unrecognizable} (by Gruenberg--Kegel graph) if there are infinitely many pairwise non-isomorphic groups having the same Gruenberg--Kegel graph as $G$.
\end{itemize}
Note that groups can be characterized by various numerical sets. For example, if we replace $\Gamma(G)$ in these definitions with $\omega(G)$, then we obtain the corresponding definitions for recognizability by spectrum. Nevertheless, if the characterization set is not specified, we suppose that it is the Grunberg-Kegel graph. It is easy to see that if $\omega(G)=\omega(H)$ for groups $G$ and $H$, then
$\Gamma(G)=\Gamma(H)$, however, the converse implication is not true in general. Consider alternating groups $A_5$ and $A_6$. Clearly, $\Gamma(A_5)=\Gamma(A_6)$, where both graphs are empty graphs on three vertices $2$, $3$, and $5$, and, on the other hand, we see that $4\in\omega(A_6)\setminus\omega(A_5)$.

\smallskip

Recently P.~J.~Cameron and the first author have proved~\cite{CaMas} that a group $G$ is almost recognizable  if and only if each group $H$ with $\Gamma(G)=\Gamma(H)$ is almost simple. Thus, the question of which almost simple groups are almost recognizable is of prime interest. At the same paper~\cite{CaMas}, a survey of known results on recognition of simple groups has been presented. Note that there are not many completed results at the moment. The situation is much better in the case of the characterization problem by spectrum, where recognition is established for many nonabelian simple groups~\cite{survey}.

\smallskip

In~\cite{Mazurov2005}, V.~D.~Mazurov conjectured that if a simple group $G$ is not isomorphic to $A_6$ and $\Gamma(G)$ has at least three connected components, then $G$ is recognizable by its spectrum. This conjecture was proved in a series of papers, the final result was obtained in~\cite{Kondrat2015}, where the author proved that groups $E_7(2)$ and $E_7(3)$ are recognizable by Gruenberg--Kegel graph, therefore, these groups are recognizable by spectrum.

Connected components of Gruenberg--Kegel graphs of simple groups were described in~\cite{Williams,Kondrat1990}. A complete result after corrections of mistakes can be found, for example, in~\cite[Tables~1--3]{ak}. If $G$ is a simple group, then $\Gamma(G)$ has at least three connected components if and only if one of the following statements holds{\rm:}

$(1)$ $G \cong G_2(q)$, where $q$ is a power of $3$, or $G\cong {^2}G_2(q)$, where $q=3^{2n+1}>3${\rm;}

$(2)$ $G \cong {^2}B_2(q)\cong Sz(q)$, where $q=2^{2n+1}>2${\rm;}

$(3)$ $G \cong F_4(q)$, where $q$ is even, or $G\cong {^2}F_4(q)$ for $q=2^{2n+1}>2${\rm;}

$(4)$ $G \cong E_8(q)${\rm;}

$(5)$ $G \cong A_1(q) \cong PSL_2(q)$, where $q>3${\rm;}

$(6)$ $G\cong {^2}D_n(3) \cong P\Omega_{2n}^-(3)$, where $n = 2^m + 1 \ge 3$ is a prime{\rm;}

$(7)$ $G$ is one of the following finite simple groups of Lie type{\rm:} ${^2}A_5(2)\cong PSU_6(2)$, $E_7(2)$, $E_7(3)$, $A_2(4) \cong PSL_3(4)$, ${^2}E_6(2)${\rm;}

$(8)$ $G$ is one of the following finite simple sporadic groups{\rm:} $M_{11}$, $M_{23}$, $M_{24}$, $J_3$, $HiS$, $Suz$, $Co_2$, $Fi_{23}$, $F_3$, $F_2$, $M_{22}$, $J_1$, $O'N$, $LyS$, $Fi'_{24}$, $F_1$, $J_4${\rm;}

$(9)$ $G \cong A_n$, where $n>6$ and both $n$ and $n-2$ are primes.

We consider the recognition problem of the simple exceptional groups of Lie type whose Gruenberg--Kegel graphs have at least three connected components. The main result of this paper is the
following statement.

\begin{theorem} Every finite simple exceptional group of Lie type,
which is isomorphic to neither ${^2}B_2(2^{2n+1})$ with $n\geq1$ nor $G_2(3)$ and whose Gruenberg--Kegel graph has at least three connected components, is almost recognizable by Grunberg--Kegel graph. Moreover, groups $ {^2}B_2(2^{2n+1})$, where $n\geq1$, and $G_2(3)$ are unrecognizable by Gruenberg--Kegel graph.
\end{theorem}

In fact, much has been known about the groups in the theorem before in the context of the recognition problem. The groups $E_7(2)$, $E_7(3)$, and ${^2}E_6(2)$ are known to be recognizable (see~\cite{Kondrat2015} and \cite{Kondrat'ev2E6(2)}). The recognizability of simple groups ${^2}G_2(q)$ with $q>3$ was established in \cite{Za06}.

If $G$ is a nonabelian simple group, then we say that $G$ is {\it quasirecognizable} by its Gruenberg--Kegel graph if every group $H$ with $\Gamma(G)=\Gamma(H)$ has a unique nonablelian composition factor $S$ and $S\cong G$.
In~\cite{Zhang_Shi_Shen}, it was proved that each group $G_2(q)$, where $q$ is an odd power of $3$, is quasirecognizable; however, this result contains an error since $\Gamma(G_2(3))=\Gamma(PSL_2(13))$. Quasirecognizability of the simple groups $F_4(q)$ and ${^2}F_4(q)$, where $q>2$ is even,  was proved in~\cite{Khosravi_Babai_2011} and~\cite{Akhlaghi_Khatami_Khosravi_2}, respectively. We prove Main Theorem for groups $G_2(q)$, where $q>3$, $F_4(q)$, and $^2F_4(q)$ in Section~\ref{F4^2F4G2Section}. The quasirecognizability of groups ${^2}B_2(q)$ was proved in~\cite{Zhang_Shi_Shen}. We consider the groups ${^2}B_2(q)$ and $G_2(3)$ in Section~\ref{B2Section}.

In~\cite{Za13}, A.~V.~Zavarnitsine proved that if $G=E_8(q)$, where $q \equiv 0, \pm 1 \pmod{5}$, and $H$ is a group such that $\Gamma(H)=\Gamma(G)$, then $H \cong E_8(u)$ for a prime power $u \equiv 0, \pm 1\pmod{5}$. In Section~\ref{E8Section}, we prove a similar result for groups $E_8(q)$, where $q\equiv \pm 2 \pmod{5}$. Therefore, as a corollary, Main Theorem holds for the groups $E_8(q)$. However, the question of quasirecognizability for the groups $E_8(q)$ is still open (see Section~\ref{E8Section} for details).

Provide a mini-survey of known results on characterization by Gruenberg--Kegel graph of the remaining simple groups whose Gruenberg--Kegel graphs have at least three connected components.

First consider groups $PSL_2(q)$, where $q=p^k$ and $p$ is a prime. The groups $PSL_2(4) \cong PSL_2(5)$, $PSL_2(7)$, $PSL_2(8)$, and $PSL_2(9) \cong A_6$ are unrecognizable while the groups $PSL_2(11)$ and $PSL_2(25)$ are $2$-recognizable (see, for example, \cite[Theorem~1.7]{CaMas}, \cite[Theorem~5]{KondrKhr2011} and \cite{Hagie_2003}). The group $PSL_2(49)$ is $5$-recognizable (see~\cite{MasNechKhis}, this result is obtained as a consequence of \cite[Theorem~5]{KondrKhr2011} and direct calculations with using Lemma~\ref{BrChar}). Let $k=1$ and $p>11$. In \cite{3Khosravi_2007_2}, it was proved that if $p \not \equiv 1 \pmod{12}$, then $G$ is recognizable, and if $p\equiv 1 \pmod{12}$, then $G$ is quasirecognizable. However, this result has at least one mistake since the group $PSL_2(13)$ is not quasirecognizable. If $p$ is odd and $k>2$ or $k=2$ and $p>7$, then $G$ is $(k,2)$-recognizable (see~\cite{BKhosravi} and \cite{ABKhosravi}). If $q>8$ is even, then $G$ is quasirecognizable (see~\cite[Theorem~3.3]{3Khosravi_2007}). In general, we suggest that the results on recognition of the groups $PSL_2(q)$ by Gruenberg--Kegel graph need a careful revision since the authors in their proofs refer to the paper~\cite{Iiyori_Yamaki} which contains a number of rather serious inaccuracies.

If  $G\cong P\Omega_{2n}^-(3)$, where $n \ge 3$ is a prime, then $G$ is recognizable (see~\cite{Ghasemabadi_Iranmanesh_Ahanjideh_2012}). The group $PSU_6(2)$ is $2$-recognizable (see~\cite[Theorem~2]{KondrRecSpor2}), the group $PSL_3(4)$ is $5$-recognizable since $\Gamma(PSL_3(4))=\Gamma(PSL_2(49))$.

The problem of recognition by Gruenberg--Kegel graph has been solved for all finite simple sporadic groups. In particular, the groups $J_1$, $M_{22}$, $M_{23}$, $M_{24}$, $Co_2$, $J_4$, $J_3$, $Suz$, $O'N$, $LyS$, $F_3$, $Fi_{23}$, $Fi'_{24}$, $F_2$, and $F_1$ are recognizable (see~\cite[Theorem~3]{Hagie_2003}, \cite[Theorem~B]{Za06}, \cite[Theorem~1]{KondrRecSpor2}, and~\cite{LeePopiel}) while the groups $M_{11}$ and $HiS$ are $2$-recognizable (see~\cite[Theorem~3]{Hagie_2003} and~\cite[Theorem~2]{KondrRecSpor2}, respectively).

Let $G \cong A_n$, where $n > 6$ and both $n$ and $n-2$ are primes. If $n=7$, then $\Gamma(G)=\Gamma(PSL_2(49))$ and, therefore, $G$ is $5$-recognizable. If $n=13$, then $G$ is recognizable (see~\cite[Lemma~24]{Staroletov2}), and if $n \ge 19$, then $G$ is known to be quasirecognizable (see \cite[Theorem~1]{Staroletov2}). The question of recognizability by Gruenberg--Kegel graph of the groups $A_n$, where $n\ge 19$ is a prime and $n-2$ is also a prime, is still open, and the conjecture~\cite[Conjecture~1]{Staroletov2} is that these groups are recognizable.

\medskip

\section{Preliminaries}

Let $n$ be an integer. Denote by $\pi(n)$ the set of all prime divisors of $n$. Let $\pi$ be a set of primes. The largest divisor $m$ of $n$ such that $\pi(m)\subseteq \pi$ is called a {\it $\pi$-part} of $n$ and is denoted by $n_\pi$. By $\pi'$ we denote the set of primes which do not belong to $\pi$. If $\pi$ consists of a unique element $p$, then we will write $n_p$ and $n_{p'}$ instead of $n_{\{p\}}$ and $n_{\{p\}'}$, respectively.

\smallskip

If $n$ is an integer and $r$ is an odd prime with $(r, n) = 1$, then $e(r, n)$ denotes the multiplicative order of $n$ modulo $r$. Given an odd integer $n$, we put $e(2, n) = 1$ if $n\equiv1\pmod{4}$, and $e(2,n)=2$ otherwise.

The following lemma is proved in~\cite{Bang}, and also in~\cite{Zs92}.
\begin{lemm}[{\rm Bang–Zsigmondy}]\label{zsigm}
Let $q$ be an integer greater than $1$. For every positive integer $m$ there exists a prime $r$ with $e(r,q)=m$ but for the cases $q=2$ and $m=1$, $q=3$ and $m=1$, and $q=2$ and $m=6$.
\end{lemm}	

Fix an integer $a$ with $|a|>1$. A prime $r$ is said to be a primitive prime divisor of $a^i-1$ if $e(r,a)=i$. We write $r_i(a)$ to denote some primitive prime divisor of $a^i-1$ if such a prime exists, and $R_i(a)$ to denote the set of all such divisors.
	
\begin{lemm}[{\rm \cite[Lemma~6]{GrechLyt}}]\label{calc}
Let $q$, $m$, and $k$ be positive integers. Then $R_{mk}(q)\subseteq R_m(q^k)$. If, in addition, $(m,k)=1$, then $R_m(q)\subset R_m(q^k)$.
\end{lemm}

\begin{lemm}\label{form}
If $r\in R_i(q)$, then $r=ik+1$, where $k$ is a positive integer.
\end{lemm}
\begin{proof} This is a consequence of Fermat's little theorem.
\end{proof}

Given a positive integer $i\neq2$, denote by $k_i(a)$ the product of all primitive prime divisors of $a^i-1$
with multiplicities counted. In case $i=2$ put $k_2(a) = k_1(-a)$. It is not difficult to verify that if $i$ is
divisible by 4 then $k_i(a)=k_i(-a)$ and if $i$ is odd then $k_i(a) = k_{2i}(-a)$. The following general formula~\cite{rotman} expresses $k_i(a)$, where $i>2$, in terms of cyclotomic polynomials:
\begin{equation}\label{eq:k_i}
k_i(a)=\frac{\Phi_i(a)}{(r, \Phi_{(i)_{r'}}(a))},
\end{equation} where $r$ is the greatest prime divisor of $i$.

\smallskip

The following assertion is well-known and its proof is elementary.
\begin{lemm}\label{zsigdiv}
Suppose that $q>1$ is an integer.
For a positive integer $i$, an odd prime $r$ divides $q^i-1$ if and only if $e(r,q)$ divides $i$.
\end{lemm}



The following lemma is a particular case of the well-known Nagell--Ljunggren equation.
\begin{lemm}[{\rm \cite{Na20}}]\label{nagell} Suppose that $x$, $y$, and $k$ are positive integers. If $x^2+x+1=y^k$, then either $k=1$ or $k=3$, $x=18$, and $y=7$. \end{lemm}

The following technical lemma follows from~Lemma~\ref{nagell}, but we give an independent proof. This statement will be needed in the proof of Main Theorem for the group $G_2(q)$.
\begin{lemm}\label{restrictions} Suppose that $q$ is an integer greater than $2$ and $\pi(q^2+\varepsilon{q}+1)=\{r\}$, where $\varepsilon\in\{+,-\}$ and $r$ is a prime. Then $r\equiv1\pmod{6}$. Moreover, if $q\equiv1\pmod8$ then either $r\equiv1\pmod{8}$ or $r\equiv3\pmod{8}$.
\end{lemm}
\begin{proof}
By assumption, there exists a positive integer $n$ such that $q^2+\varepsilon{q}+1=r^n$.
Clearly, $r$ is odd. Note that $q^2+\varepsilon{q}+1$ is not divisible by 9, so if $r=3$, then $n=1$ and $q\in\{1,2\}$. If $r\neq3$, then $r$ divides $k_3(\varepsilon{q})=\frac{q^2+\varepsilon{q}+1}{(q-\varepsilon1,3)}$ and hence $r\in R_3(\varepsilon{q})$. It follows from Lemma~\ref{form} that $r\equiv1\pmod{6}$.

Suppose that $q\equiv1\pmod{8}$. Then either $q^2+\varepsilon{q}+1\equiv1\pmod8$ or
$q^2+\varepsilon{q}+1\equiv3\pmod8$. Therefore, if $r\equiv5\pmod{8}$ or $r\equiv7\pmod{8}$, then
$n$ is even. On the other hand, $(q-1)<q^2-q+1<q^2$ and $q^2<q^2+q+1<(q+1)^2$, so $n$ cannot be even.
This implies that either $r\equiv1\pmod{8}$ or $r\equiv3\pmod{8}$.
\end{proof}

Let $G$ be a finite group. Denote the number of connected components of $\Gamma(G)$
by $s(G)$, and the set of connected
components of $\Gamma(G)$ by $\{\pi_i(G) \mid 1 \leq i \leq s(G) \}$; for a group $G$
of even order, we assume that $2 \in \pi_1(G)$. Denote by $t(G)$ the \emph{independence number} of $\Gamma(G)$, that is the greatest size of a coclique  (i.\,e. induced subgraph with no edges) in $\Gamma(G)$. If $r\in\pi(G)$, then denote by $t(r,G)$ the greatest size of a coclique in $\Gamma(G)$ containing $r$.

\begin{lemm}\label{NormalSeriesAdj} Let $A$ and $B$ be normal subgroups of a group $G$ such that $A\le B$.
If $r, s \in \pi(B/A)\setminus (\pi(A) \cup \pi(G/B))$, then $r$ and $s$ are adjacent in $\Gamma(G)$ if and only if $r$ and $s$ are adjacent in $\Gamma(B/A)$.
\end{lemm}

\begin{proof} The proof of this lemma is elementary.
\end{proof}

\begin{lemm}[{\rm \cite{Va05}}]\label{vas}
Let $G$ be a finite group with $t(G)\geq3$ and $t(2,G)\geq2$. Then the following statements hold.
	
$(1)$ There exists a nonabelian simple group $S$ such that $S \unlhd \overline{G} = G/K \le\operatorname{Aut}(S)$,  where $K$ is the solvable radical of $G$ {\rm(}i.\,e., the largest solvable normal subgroup of $G${\rm)}.
	
$(2)$ For every coclique $\rho$ of $\Gamma(G)$ of size at least three, at most one prime in $\rho$ divides the product $|K|\cdot|\overline{G}/S|$. In particular, $t(S)\geq t(G)-1$.
	
$(3)$ One of the following two conditions holds:	
	
$\mbox{ }$$\mbox{ }$$\mbox{ }$$(3.1)$ $S\cong A_7$ or $L_2(q)$ for some odd $q$, and $t(S)=t(2,S)=3$.
	
$\mbox{ }$$\mbox{ }$$\mbox{ }$$(3.2)$
Every prime $p\in\pi(G)$ nonadjacent to $2$ in $\Gamma(G)$ does not divide the product $|K|\cdot|\overline{G}/S|$. In particular, $t(2,S)\geq t(2,G)$.
\end{lemm}

\begin{lemm}[{\rm \cite[Lemma 1]{Staroletov2}}]\label{semid}
Let $N\unlhd G$ be an elementary abelian subgroup and $H = G/N$. Define a homomorphism $\phi: H\rightarrow\operatorname{Aut}(N)$ as follows $n^{\phi(gN)}=n^g$. Then $\Gamma(G)=\Gamma(N\rtimes_{\phi}H)$.	
\end{lemm}

\begin{lemm}[{\rm \cite[Proposition~3.1]{CaMas}}]\label{NumSimple} Let $\pi$ be a finite set of primes. The number of pairwise nonisomorphic nonabelian simple groups $S$ with $\pi(S)\subseteq\pi$ is finite, and is at most $O(|\pi|^3)$.
\end{lemm}

Let $S$ be a finite simple group of Lie type in characteristic $p$. Let $A$ be any abelian $p$-group with an $S$-action. Any element $s \in S$ is said to be {\it unisingular} on $A$ if $s$ has a (nonzero) fixed point on $A$. The group $S$ is said to be {\it unisingular} if every element $s \in S$ acts unisingularly on every finite abelian $p$-group $A$ with an $S$-action.
Denote by $PSL^{\varepsilon}_n(q)$, where $\varepsilon \in\{+,-\}$, the group $PSL_n(q)$ if $\varepsilon=1$ and
$PSU_n(q)$ if $\varepsilon=-1$. Similarly, $E^{\varepsilon}_6(q)$ denotes the simple group $E_6(q)$ if $\varepsilon=1$ and ${}^2E_6(q)$ if $\varepsilon=-1$.

\begin{lemm}[{\rm\cite[Theorem 1.3]{GuTi03}}]\label{Unisingular} A finite simple group $S$ of Lie type of characteristic $p$ is unisingular if
and only if $S$ is one of the following{\rm:}

$(i)$ $PSL_n^\varepsilon(p)$ with $\varepsilon \in \{+,-\}$ and $n$ divides  $p-\varepsilon 1${\rm;}

$(ii)$ $P\Omega_{2n+1}(p)$, $PSp_{2n}(p)$ with $p$ odd{\rm;}

$(iii)$ $P\Omega_{2n}^\varepsilon (p)$ with $\varepsilon \in \{+,-\}$, $p$ odd, and $\varepsilon =(-1)^{n(p-1)/2}${\rm;}

$(iv)$ ${^2}G_2(q)$, $F_4(q)$, ${^2}F_4(q)$, $E_8(q)$ with $q$ arbitrary{\rm;}

$(v)$ $G_2(q)$ with $q$ odd{\rm;}

$(vi)$ $E_6^\varepsilon(p)$ with $\varepsilon \in \{+,-\}$ and $3$ divides $p-\varepsilon 1${\rm;}

$(vii)$ $E_7(p)$ with $p$ odd.
\end{lemm}

\begin{lemm}[{\rm \cite[Proposition~2]{Za13}}]\label{l:r24}
Let $G={}^3D_4(q)$ act on a nonzero vector space $V$ over a field of characteristic not dividing $q$ {\rm(}possibly, zero{\rm)}. Then each element of $G$ of order $q^4-q^2+1$ fixes on $V$ a nonzero vector.
\end{lemm}

The following lemma is also well-known, but we provide its proof for completeness.

\begin{lemma}\label{Eigenvector1} Let $G$ be a group, $g\in G$ an element of order $r$, and $\phi$  a non-trivial irreducible representation of $G$ on a nonzero vector space $V$. If the minimum polynomial degree of $\phi(g)$ equals to $r$, then $g$ fixes on $V$ a nonzero vector.
\end{lemma}

\begin{proof} Let $A=\phi(g)$. Since $g^r=1$, we have $A^r=1$, and, therefore, the minimal polynomial for $A$ divides the polynomial $x^r-1$. Since the minimum polynomial degree of $A$ equals to $r$, we have that the minimum polynomial for $A$ is $x^r-1$.

By the Cayley--Hamilton theorem, $A$ is a root of its characteristic polynomial. In particular, $1$ is an eigenvalue of $A$, and, therefore, each eigenvector of $A$ which corresponds to the eigenvalue~$1$, is fixed by $A=\phi(g)$.
\end{proof}

\begin{lemm}[{\rm \cite[Theorem 1.1]{TiepZal}}]\label{TiepZalThm}
Let $G$ be one of the groups ${^2}B_2(q)$, where $q > 2$, ${^2}G_2(q)$, where $q>3$, ${^2}F_4(q)$, $G_2(q)$, ${^3}D_4(q)$. Let $g \in G$ an element of prime power order coprime to $q$. Let $\phi$ be a non-trivial irreducible representation of $G$ over a field $F$ of characteristic $l$ coprime to $q$. Then the minimum polynomial degree of $\phi(g)$
equals $|g|$, unless possibly when $G = {^2}F_4(8)$, $l= 3$, $p=109$ and $\phi(1) < 64692$.
\end{lemm}

\begin{lemm}[{\rm \cite[Lemma~4]{DoJaLu}}]\label{BrChar} Let $G$ be a finite simple group, $F$ a field of characteristic $p > 0$, $V$ an absolute irreducible $GF$-module, and $\beta$ a Brauer character of $V$. If $g \in G$ is an element of prime order distinct from $p$, then
$$
dim C_V (g) = (\beta_{\langle g \rangle}, 1_{\langle g \rangle}) = \frac{1}{|g|}\sum_{x \in {\langle g \rangle}} \beta(x).
$$
\end{lemm}

\begin{lemm}\label{E8nontrivK_R24}
Let $G$ be a group with a non-trivial nilpotent normal subgroup $K$ such that $G/K$ has a subgroup $H$ isomorphic to $E_8(q)$, where $q$ is a prime power. Then $R_{24}(q) \subset \pi_1(G)$.
\end{lemm}

\begin{proof} Note that $R_{24}(q) \subset \pi(H)$.

Let $q=p^l$, where $p$ is a prime. Suppose that $K$ is a $p$-group.
Factoring $G$ and $K$ by $K'$, we can assume that $K$ is abelian.
According to \cite[Table~3]{ak}, $p$ lies in $\pi_1(G)$. By Lemma~\ref{Unisingular}, $H$ is unisingular and hence $p$ is adjacent to each element of $\pi(H)$. Therefore, $R_{24}(q) \subset \pi_1(G)$.

Take any $r\in\pi(K)$. Since Sylow $2$-subgroups of $H$ are non-cyclic, we infer that either $r=2$ or $2$ and $r$ are adjacent in $\Gamma(G)$ (see, for example, \cite[Theorem~10.3.1]{Gorenstein}). Therefore, $r\in\pi_1(G)$.

Suppose that there exists $r\in \pi(K) \cap R_{24}(q)$. Since $\pi(K)$ is a clique in $\Gamma(G)$ and $r\in\pi_1(G)$, we get that $R_{24}(q) \subset \pi_1(G)$.

If $K$ is not a $p$-group and $\pi(K) \cap R_{24}(q)=\varnothing$, then we can choose $r\in\pi(K)\setminus(\{p\}\cup R_{24}(q))$. By Lemma~\ref{l:r24}, $r$ is adjacent in $\Gamma(G)$ to any prime from $R_{24}(q)$. This implies that $R_{24}(q) \subset \pi_1(G)$.
\end{proof}


\section{The Grunberg--Kegel graphs of some exceptional groups of Lie type}

A criterion for the adjacency of vertices in the Gruenberg--Kegel graph for all finite nonabelian simple groups was obtained in \cite{VaVd05}. Based on this paper and \cite{VaVd11}, in this section, we collect the necessary information for exceptional groups of Lie type from Main Theorem.

By the compact form of the Gruenberg--Kegel graph $\Gamma(G)$ for a group $G$ we mean a graph whose vertices are labeled with sets of primes. A vertex labeled by a set $\tau$ represents the clique of $\Gamma(G)$ such that every vertex in this clique labeled by a prime from $\tau$. An edge
connecting two sets represents the set of edges of $\Gamma(G)$ that connect each vertex in the first set with each vertex in the second.

\begin{lemm}[{\rm \cite[Proposition~2.7]{VaVd11} and \cite[Proposition~3.2]{VaVd05}}]\label{graphF4}
	Let $G\cong F_4(q)$, where $q=2^n$, $r,s\in \pi(G)$ and $r\ne s$. Then $r$ and $s$ are nonadjacent if and only if one of the following conditions holds {\rm(}up to permutation{\rm):}
	\begin{enumerate}
  \item [\text{\normalfont(1)}] $2=r$, and $e(s,q)\in\{8,12\}$.
  \item [\text{\normalfont(2)}] $s,r\ne2$, $k=e(r,q),~l=e(s,q)$, $1\le k < l$ and either $l\in \{8, 12\}$, or $l = 6$ and $k\in \{3, 4\}$, or $l = 4$ and $k = 3$.
    \end{enumerate}

In particular, the compact form for $\Gamma(F_4(q))$ is the following.

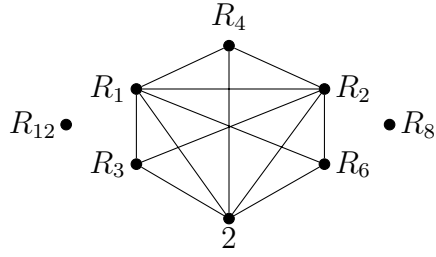
\centering{
	\begin{tikzpicture}
		\tikzstyle{every node}=[draw,circle,fill=black,minimum size=4pt,
		inner sep=0pt]
		
		\draw (0,0) node (1) [label=left:$R_1$]{}
		++ (0:2.5cm) node (2) [label=right:$R_2$]{}
		++ (270:1.0cm) node (6) [label=right:$R_6$]{}
	    ++ (180:2.5cm) node (3) [label=left:$R_3$]{}
	    ++ (52:2.0cm) node (4) [label=above:$R_4$]{}
	    ++ (270:2.3cm) node (p) [label=below:$2$]{}
	    ++ (150:2.5cm) node (12) [label=left:$R_{12}$]{}
	    ++ (0:4.3cm) node (8) [label=right:$R_{8}$]{}
	
		(1)--(2)
		(2)--(6)
		(1)--(6)
		(2)--(3)
		(1)--(3)
		(1)--(4)
		(2)--(4)
		(1)--(p)
		(2)--(p)
		(3)--(p)
		(6)--(p)
		(4)--(p)
		;
	\end{tikzpicture}
}
\end{lemm}

\begin{remark} Note that $R_6(2)$ and $R_1(2)$ are empty sets.
\end{remark}

\begin{lemm}[{\rm \cite[Proposition 3.3]{VaVd05} and \cite[Proposition 2.9]{VaVd11}; see also \cite[Lemma~3]{DengShi}}]\label{graph2F4}
	Let $G\cong {}^2F_4(q)$, where $q=2^{2n+1}$, $r,s\in \pi(G)$ and $r\ne s$. Put $m_1(n)=q-1$, $m_2(n)=q+1$, $m_3(n)=q^2+1$, $m_4(n)=q^2-q+1$, $m_5(n)=q^2-\sqrt{2q^3}+q-\sqrt{2q}+1$, $m_6(n)=q^2+\sqrt{2q^3}+q+\sqrt{2q}+1$. Then $r$ and $s$ are nonadjacent in $\Gamma(G)$ if an only if one of the following conditions holds{\rm:}
\begin{enumerate}
  \item [\text{\normalfont(1)}] $2=r$, $s$ divides $m_k(n)$, $s\ne3$, and $k > 3$.
  \item [\text{\normalfont(2)}] $2\ne s,r$; either $3 \not = r\in \pi(m_k(n))$, $3\not = s\in \pi(m_l(n))$ for $k\ne l$, and $\{k,l\}\ne \{1, 2\},~\{1, 3\}$; or $r = 3$ and $s\in \pi(m_l(n))$, where $l\in \{3,5,6\}$.
    \end{enumerate}
In particular, the compact form for $\Gamma({}^2F_4(q))$ is the following.

\centering{
	\begin{tikzpicture}
		\tikzstyle{every node}=[draw,circle,fill=black,minimum size=4pt,
		inner sep=0pt]
		
		\draw (0,0) node (1) [label=right:$2$]{}
         ++ (90:1.7cm) node (7)
        [label=right:$\pi(q+1)\!\setminus\!\{3\}$]{}
        ++ (180:1.7cm) node (5) [label=left:$3$]{}
        ++ (270:1.7cm) node (4) [label=left:$\pi(q-1)$]{}
        ++ (68:2.4cm) node (6) [label=left:$\pi(q^2-q+1)\!\setminus\!\{3\}$]{}
		++ (-90:2.8cm) node (2) [label=right:$\pi(q^2+1)$]{}
        ++ (45:2.1cm) node (8)       [label=right:$\pi(q^2+\sqrt{2q^3}+q+\sqrt{2q}+1)$]{}
        ++ (180:3.1cm) node (9)
        [label=left:$\pi(q^2-\sqrt{2q^3}+q-\sqrt{2q}+1)$]{}
		(1)--(2)
		(1)--(4)
		(4)--(5)
		(5)--(6)
        (4)--(7)
		(1)--(7)
		(5)--(7)
        (5)--(1)
        (2)--(4)
		;
	\end{tikzpicture}
}
\end{lemm}

\begin{lemm}[{\rm  \cite[Proposition~2.7]{VaVd11} and \cite[Propositions~3.2,~4.5]{VaVd05}}]\label{graphG2}

	Let $G\cong G_2(q)$, where $q=3^k$, $r,s\in \pi(G)$ and $r\ne s$. Then $r$ and $s$ are nonadjacent in $\Gamma(G)$ if and only if $e(r, q)$ or $e(s, q)\in\{3,6\}$.

    In particular, the compact form for $\Gamma(G_2(q))$ is the following.

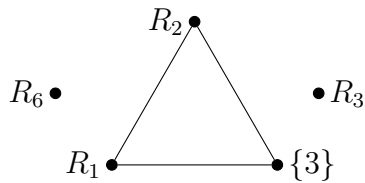
\centering{
	\begin{tikzpicture}
		\tikzstyle{every node}=[draw,circle,fill=black,minimum size=4pt,
		inner sep=0pt]
		
		\draw (0,0) node (1) [label=left:$R_1$]{}
		++ (60:2.2cm) node (2) [label=left:$R_2$]{}
		++ (-60:2.2cm) node (3) [label=right:$\{3\}$]{}
		++ (60:1.1cm) node (4) [label=right:$R_3$]{}
		++ (180:3.5cm) node (5) [label=left:$R_6$]{}
		(1)--(2)
		(1)--(3)
		(2)--(3)
		;
	\end{tikzpicture}
}
\end{lemm}

\begin{remark} Note that the set $R_1(3)$ is empty.
\end{remark}

\begin{lemm}[{\rm \cite[Proposition~2.7]{VaVd11} and \cite[Propositions~3.2,~4.5]{VaVd05}}]\label{graphE8}
Let $G\cong  E_8(q)$, where $q$ is a power of a prime $p$. Suppose that $r,s\in \pi(G)$ with $r\ne s$. Then $r$ and $s$ are nonadjacent in $\Gamma(G)$ if and only if one of the following conditions holds:
	\begin{enumerate}
  \item [\text{\normalfont(1)}] $r\in\{2,p\}$, $s\ne p$ and $e(s,q)\in\{15,20,24,30\}$.
  \item [\text{\normalfont(2)}] $s,r\not\in\{2,p\}$, $k=e(r,q),~l=e(s,q)$, $1\le k < l$, and either $l = 6$ and $k = 5$, or $l\in \{7, 14\}$ and $k\ge 3$, or $l = 9$ and
$k\ge 4$, or $l\in \{8, 12\}$ and $k\ge 5$, $k\ne 6$, or $l = 10$ and $k\ge 3$, $k\not\in\{4,6\}$, or $l = 18$ and
$k\not\in\{1, 2, 6\}$, or $l = 20$ and $r\cdot k\ne20$, or $l\in \{15, 24, 30\}$.
    \end{enumerate}

In particular, the compact form for $\Gamma(E_8(q))$ is the following. Here, $R (q)= R_1(q) \cup R_2(q) \cup \{p\}$ and the vector from $5$ to $R_4(q)$ and the dotted edge $\{5, R_{20}(q)\}$ indicate that $R_4(q)$ and $R_{20}(q)$ are not adjacent, but if $5\in R_4(q)$ {\rm(}i.e., $q^2\equiv -1 \pmod {5}${\rm)}, then there exist edges between $5$ and the primes from $R_{20}(q)$.
\usetikzlibrary {arrows.meta}
 
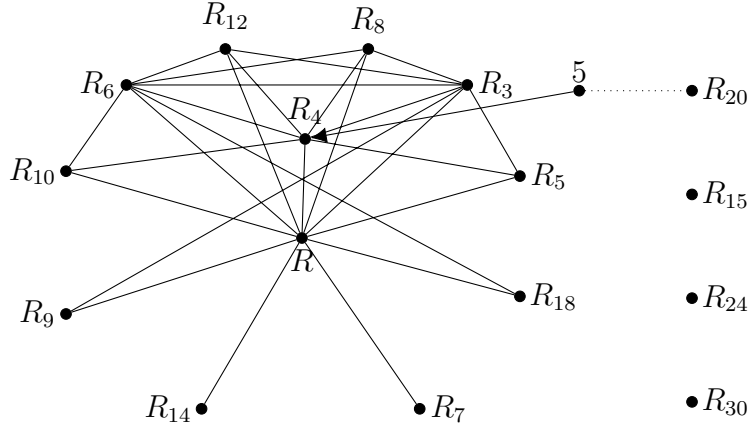
\centering{
	\begin{tikzpicture}
		\tikzstyle{every node}=[draw,circle,fill=black,minimum size=4pt,
		inner sep=0pt]
		
		\draw (0,0) node (1) [label=below:$R$]{}
		++ (-15:3.0cm) node (18) [label=right:$R_{18}$]{}
        ++ (90:1.6cm) node (5) [label=right:$R_{5}$]{}
        ++ (120:1.4cm) node (3) [label=right:$R_{3}$]{}
        ++ (160:1.4cm) node (8) [label=above:$R_{8}$]{}
        ++ (180:1.9cm) node (12) [label=above:$R_{12}$]{}
        ++ (200:1.4cm) node (6) [label=left:$R_{6}$]{}
        ++ (-125:1.4cm) node (10) [label=left:$R_{10}$]{}
        ++ (-90:1.9cm) node (9) [label=left:$R_{9}$]{}
        ++ (-35:2.2cm) node (14) [label=left:$R_{14}$]{}
        ++ (0:2.9cm) node (7) [label=right:$R_{7}$]{}
        ++ (113:3.9cm) node (4) [label=above:$R_{4}$]{}
        ++ (10:3.7cm) node (55) [label=above:$5$]{}
        ++ (0:1.5cm) node (20) [label=right:$R_{20}$]{}
		++ (-90:1.38cm) node (15) [label=right:$R_{15}$]{}
		++ (-90:1.38cm) node (24) [label=right:$R_{24}$]{}
		++ (-90:1.38cm) node (30) [label=right:$R_{30}$]{}
        (1)--(7)
		(1)--(14)
		(4)--(8)
		(10)--(6)
		(10)--(4)
		(10)--(1)
		(1)--(6)
		(6)--(4)
		(6)--(18)
		(6)--(8)
		(6)--(12)
		(6)--(3)
		(1)--(18)
		(1)--(8)
		(3)--(8)
		(1)--(12)
		(4)--(12)
		(3)--(12)
		(1)--(9)
		(3)--(9)
		(1)--(3)
		(4)--(3)
		(5)--(3)
		(5)--(1)
            (1)--(4)
		(5)--(4);
        \draw[arrows = {-Latex[width=6pt, length=6pt]}]
		(55)--(4);
		\draw[dotted]
		(55)--(20)
		;
	\end{tikzpicture}}
\end{lemm}

\section{Almost recognizability of groups $G_2(q)$, $F_4(q)$, and ${^2}F_4(q)$ by Gruenberg--Kegel graph}\label{F4^2F4G2Section}

In this section, we prove Main Theorem for groups $L=F_4(q)$, where $q \ge 2$ is a power of $2$, $L={^2}F_4(q)$, where $q=2^{2m+1}>2$, and $L=G_2(q)$, where $q>3$ is a power of $3$.

\begin{T}\label{F42F4_AlmRec} If $G$ is a group with $\Gamma(G)=\Gamma(L)$, then $L \cong \operatorname{Inn}(L) \unlhd G \le \operatorname{Aut}(L)$. {In particular}, $L$ is almost recognizable by Gruenberg--Kegel graph.
\end{T}


By Lemmas~\ref{graphF4}, \ref{graph2F4}, and \ref{graphG2}, we see that $t(L)\geq3$ and $t(2,L)\geq2$. It follows from Lemma~\ref{vas} that there exists a nonabelian simple group $S$ such that $S \cong \operatorname{Inn}(S) \leq G/K\leq \operatorname{Aut}(S)$, where $K$ is the solvable radical of $G$. Moreover, by the Thompson theorem on finite groups with fixed-point-free automorphisms of prime order~\cite[Theorem~1]{Thompson}, $K$ is nilpotent. To prove Theorem~\ref{F42F4_AlmRec}, it suffices to show that $S\cong L$ and $K=1$. These two facts are established in the following four lemmas.

\begin{lemm}\label{quasiF42F4} $S \cong L$.
\end{lemm}

\begin{proof}

If $q>2$ is even and $L={^2}F_4(q)$ or $L=F_4(q)$, then Lemma follows from~\cite[Theorem~3.4]{Akhlaghi_Khatami_Khosravi_2} and~\cite[Theorem~3.3]{Khosravi_Babai_2011}, respectively.

Suppose that $L=F_4(2)$. Then $\pi(S)\subseteq\pi(L)=\{2,3,5,7,13,17\}$.
Note that $13\in R_{12}(2)$ and $17\in R_{8}(2)$.
By Lemma~\ref{graphF4}, we find that $13$ and $17$ are nonadjacent to all vertices in $\Gamma(G)$. Therefore, $\{13,17\}\subset \pi(S)$ by Lemma~\ref{vas}.
Inspecting~\cite[Table 1]{Za09}, we see that $S\in\{PSU_4(4)$, $PSU_{3}(17)$, $PSL_2(13^2)$, $PSp_4(13)$, $PSL_3(16)$, $PSp_6(4)$, $P\Omega_{8}^{+}(4)$, $F_4(2)\}$.
Using \cite[Corollary~3]{But_lin}, we find that $5$ and $17$ are adjacent in
$\Gamma(PSL_2(13^2))$ and $\Gamma(PSL_3(16))$, while $3$ and $17$ are adjacent in $\Gamma(PSU_3(17))$ and $\Gamma(PSU_4(4))$.
According to \cite[Corollaries~2--4]{But_ort}, 5 and 17 are adjacent in $\Gamma(PSp_4(13))$,
$\Gamma(PSp_6(4))$, and $\Gamma(P\Omega_{8}^{+}(4))$. This implies that $S\cong F_4(2)$, as claimed.

It remains to consider the case $L\cong G_2(q)$, where $q>3$ is a power of 3.
If $q>3$ is an odd power of $3$, then $S\cong L$ by~\cite[Theorem~1.1]{Zhang_Shi_Shen}. Suppose that $L\cong G_2(q)$, where $q=3^{2k}$ for a positive integer $k$. By Lemma~\ref{graphG2}, sets $R_{3}(q)$ and $R_{6}(q)$ are connected components in $\Gamma(L)$.  It follows from Lemma~\ref{vas} that sets $R_{3}(q)$ and $R_{6}(q)$ are connected components in $\Gamma(S)$ and hence $\Gamma(S)$ has at least three connected components. Therefore, $S$ isomorphic to a group listed in Introduction before Main Theorem.  We show that $S\cong L$ considering each case for $S$ separately.
We will extensively use that $k_3(q)=\frac{q^2+q+1}{(q-1,3)}=q^2+q+1$ and $k_6(q)=\frac{q^2-q+1}{(q+1,3)}=q^2-q+1$ (see equation~(\ref{eq:k_i})).

{\it Case $S\cong A_p$, where $p>6$ and both $p$ and $p-2$ are primes}. The connected components of $\Gamma(S)$ are $\pi_1(S)$, $\{p\}$, and $\{p-2\}$. We know that $R_3(q)$ and $R_6(q)$ are connected components in $\Gamma(S)$, so either $p-2\in R_3(q)$ or $p-2\in R_6(q)$. It follows from Lemma~\ref{form} that $p-3$ is divisible by $3$; a contradiction since $p\neq3$.

{\it Case $S\cong E_8(u)$, where $u$ is a prime power}.
Since $3^{2k}\equiv(-1)^k\pmod{5}$, we infer that $q \equiv \pm 1 \pmod{5}$.
This implies that there exists $\varepsilon\in\{+,-\}$ such that
$q^2+\varepsilon{q}+1 \equiv 3 \pmod{5}$.
Take any prime divisor $r$ of $k_3(\varepsilon{q})=q^2+\varepsilon{q}+1$
such that $r \not \equiv 1 \pmod{5}$. Since $r\in R_3(\varepsilon{q})$,
Lemma~\ref{graphE8} implies that $r\in R_j(u)$, where $j\in\{15,20,24,30\}$.
Then $r-1$ is divisible by $j$ and hence $j=24$. Since $R_{24}(u)$ is a connected component in $\Gamma(S)$, we find that $R_3(\varepsilon{q})=R_{24}(u)$.
If $K\neq 1$, then Lemma~\ref{E8nontrivK_R24} implies that $R_{24}(u)\subseteq\pi_1(G)$; a contradiction since $r\not\in\pi_1(G)$. Assume that there exists $s\in\pi(\overline{G}/S)$ such that $s>3$. It follows from Lemma~\ref{vas} that $s\in\pi_1(G)$. By \cite[Theorem~2.5.12, Definition~2.5.13]{GoLySo98}, there exists an element $g\in G \setminus S$ such that $|g|=s$ and $g$ acts on $S$ as a field automorphism. It is known that $C_S(g)\cong E_8(u^{1/s})$ (see, e.g.,~\cite[Proposition~4.9.1]{GoLySo98}). By Lemma~\ref{calc}, we see that $R_{24}(u^{1/s})\subseteq R_{24}(u)$ and hence $s$ is adjacent to $r$ in $\Gamma(G)$; we arrive at a contradiction since $s\in\pi_1(G)$ and $r\not\in\pi_1(G)$.
Now consider a coclique $\rho=\{s_1,s_2,\ldots, s_{12}\}$ of maximal size in $\Gamma(S)$. By Lemma~\ref{graphE8}, we have $(s_i,6)=1$ for every $i\in\{1,\ldots,12\}$. Applying Lemma~\ref{NormalSeriesAdj}, we find that $t(G) \ge 12$; a contradiction with Lemma~\ref{graphG2}.



{\it Cases $S\in \{PSU_6(2), PSL_3(4), M_{11}, M_{23}, M_{24}, J_3, HiS, Suz, Co_2, Fi_{23}, F_2, M_{22}, Fi_{24}', F_1\}$}.
By Lemma~\ref{vas}, we infer that $\{2,r_3(q),r_6(q)\}$ is a coclique in $\Gamma(S)$.
It follows from Lemma~\ref{form} that $r_3(q)\equiv 1\pmod{6}$ and $r_6(q)\equiv 1\pmod{6}$.
Using \cite{Atlas}, we see that there is no such a coclique in $\Gamma(S)$; contradiction.

{\it Case $S\cong F_4(u)$, where $u$ is even}. Since $q^2+q+1\equiv3\pmod{8}$, there exists   $r\in R_3(q)$ such that $r\not\equiv 1\pmod{8}$.
By Lemmas~\ref{vas},~\ref{graphF4}, and \ref{form}, we infer that $r\in R_{12}(u)$.
Then $\pi(q^2+q+1)=R_{12}(u)$. On the other hand, we have $q^2+q+1\equiv 7\pmod{12}$ and hence there exists $s\in\pi(q^2+q+1)$ such that $s\not\equiv1\pmod{12}$; a contradiction with Lemma~\ref{form}.

{\it Case $S\cong {}^2F_4(u)$, where $u=2^{2m+1}>2$}. By Lemma~\ref{graph2F4}, it is true that $\pi_2(S)\cup\pi_3(S)=\pi(u^4-u^2+1)\subseteq \pi(u^6+1)$. Consider any prime $r\in R_3(q)$. Then
$r\in\pi_2(S)\cup\pi_3(S)$ and $r$ divides $(u^3)^2+1$. Since $(u^3)^2\equiv -1\pmod{r}$, we find that $-1$ is a quadratic residue modulo $r$. This implies that $r\equiv1\pmod{4}$.
Since $r$ is an arbitrary element of $R_3(q)$, we infer that $q^2+q+1\equiv 1\pmod{4}$.
On the other hand, $q^2+q+1 \equiv3^{4k}+3^{2k}+1\equiv3\pmod{4}$; a contradiction.

{\it Case $S\cong {}^2B_2(u)$, where $u=2^{2m+1}>2$.} According to \cite[Table~3]{ak}, we can assume that $\pi_2(S)=\pi(u-1)$, $\pi_3(S)=\pi(u-\sqrt{2u}+1)$, and $\pi_4(S)=\pi(u+\sqrt{2u}+1)$. Note that $\pi(u-\sqrt{2u}+1)\cup\pi(u+\sqrt{2u}+1)=\pi(u^2+1)$. If $R_3(q)\neq\pi_2(S)$, then
$R_3(q)\subseteq\pi(u^2+1)$ and we get a contradiction arguing as in the case $S\cong {}^2F_4(u)$.
Therefore, we can assume that $R_3(q)=\pi(u-1)$.  Suppose that $r~|~u-1=2^{2m+1}-1$. Then $r~|~2^{2n+2}-2$ and hence $2$ is a quadratic residue modulo $r$. This implies that $r\equiv \pm1\pmod{8}$. It follows that $q^2+q+1\equiv\pm1\pmod8$; a contradiction since
$q^2+q+1\equiv3^{4k}+3^{2k}+1\equiv3\pmod8$.

{\it Case $S\cong {^2}G_2(u)$, where $u=3^{2m+1}>3$}. According to \cite[Table~2]{ak}, we can assume that $\pi_2(S)=\pi(u-\sqrt{3u}+1)$ and $\pi_3(S)=\pi(u+\sqrt{3u}+1)$. Take any prime $r\in R_{12k}(3)$.
By Lemma~\ref{calc}, we infer that $r\in R_6(q)$. We know that $R_6(q)\subseteq\pi_2(S)\cup\pi_3(S)=\pi(u^2-u+1)=R_6(u)$. By Lemma~\ref{zsigdiv}, we conclude that $12k~|~6\cdot(2m+1)$; a contradiction.

{\it Case $S\cong {^2}D_p(3)$, where $p=2^{m}+1\geq3$ is prime}. According to \cite[Table~2]{ak}, we can assume that $\pi_2(S)=\pi((3^{p-1}+1)/2)$ and $\pi_3(S)=\pi((3^{p}+1)/4)$. Take any $r\in R_{12k}(3)$. By Lemma~\ref{calc}, we infer that $r\in R_6(q)$. Since $r\in\pi_2(S)\cup\pi_3(S)$,
we find that $r$ divides $3^{2(p-1)}-1$ or $3^{2p}-1$. It follows from Lemma~\ref{zsigdiv} that $12k~|~2(p-1)$ or $12k~|~2p$; a contradiction.

{\it Case $S\cong A_1(u) \cong PSL_2(q)$, where $u=2^m>2$}. According to \cite[Table~2]{ak}, we can assume that $\pi_2(S)=\pi(u-1)$ and $\pi_3(S)=\pi(u+1)$. Note that 3 divides $u^2-1$.
On the other hand, $\pi(u^2-1)=\pi_2(S)\cup\pi_3(S)=\pi_2(G)\cup\pi_3(G)$; a contradiction since $3\in\pi_1(G)$.

{\it Case $S\cong A_1(u)\cong PSL_2(u)$, where $3<u=v^n$ and $u$ is odd}. Consider $\varepsilon\in\{+,-\}$ such that $u\equiv\varepsilon 1\pmod4$. According to~\cite[Table~2]{ak}, we can assume that $\pi_1(S)=\pi(u-\varepsilon 1)$, $\pi_2(S)=\{v\}$, and $\pi_3(S)=\pi(\frac{u+\varepsilon 1}{2})$. Therefore, there exists $\tau\in\{+,-\}$ such that $\pi(q^2+\tau{q}+1)=\{v\}$ and $\pi(q^2-\tau{q}+1)=\pi(\frac{u+\varepsilon 1}{2})$.
Moreover, we know that $\pi(u-\varepsilon 1)\subseteq\pi(3(q^2-1))$. Since $q^2-q+1=(q-1)^2+(q-1)+1$,
Lemma~\ref{nagell} implies that either $q^2+\tau{q}+1=v$
or $q=19$ and $v=7$. By assumption, $q=3^{2k}$ and hence
$q^2+\tau{q}+1=v$. Then $v-1$ is divisible by $9$.
Therefore, $v\geq19$ and $3\in\pi(\frac{u-1}{2})\setminus\pi(q^2-\tau{q}+1)$. This implies that $\varepsilon=+$ since, by Lemma~\ref{graphG2}, $3 \in \pi_1(G)$.

Suppose that $n$ is even. Then $v^2-1$ divides $u-1$,
so $\pi(v+1)\subseteq\pi(q^2-1)$. Take any $r\in\pi(v+1)$. Since $q\equiv\pm1\pmod{r}$ and $r$ divides $q^2+\tau{q}+2$, we find that $r$ divides $3\pm1$ and hence $r=2$. Therefore,
$q^2+\tau{q}+2$ is a power of 2. On the other hand,
$q^2+\tau{q}+2\equiv1\pm1+2\pmod{8}$. This implies that $q^2+\tau{q}+2\leq 4$; a contradiction.

We can assume that $n$ is odd.
Then $\frac{v+1}{2}$ divides $\frac{u+1}{2}$. Take any
$r\in\pi(\frac{v+1}{2})$. Then $r$ divides both $q^2+\tau{q}+2$ and $q^2-\tau{q}+1$. Therefore,
$2\tau{q}\equiv-1\pmod{r}$ and hence $0\equiv 4q^2+4\tau{q}+8\equiv 4\tau{q}+9\equiv7\pmod{r}$.
This implies that $r=7$ and $v+1=2\cdot 7^m$ for a positive integer $m$. Since $q\equiv1\pmod8$
and $q^2+\tau{q}+2=2\cdot 7^m$, we infer that $3+\tau1\equiv 2\cdot(-1)^m\pmod{8}$.
This implies that $\tau=-1$ and $m$ is even. A straightforward calculation shows that $9^{2k}-9^k+2$ is not divisible by 49 for all positive integer $k$; a contradiction.

{\it Cases $S\in \{F_3, O'N, J_1, LyS, J_4, {}^2E_6(2), E_7(2), E_7(3)\}$}. According to \cite[Tables~2, 3]{ak}, we see that there exist primes $r$ and $s$ such that $\pi(q^2+q+1)=\{r\}$ and $\pi(q^2-q+1)=\{s\}$.
Applying Lemma~\ref{restrictions}, we find that $r\equiv s\equiv1\pmod{6}$
and the remainders of $r$ and $s$ divided by 8 belong to the set $\{1,3\}$.
Inspecting \cite[Tables~2, 3]{ak}, we see that in each case there are no two primes satisfying these restrictions; a contradiction.

Thus, we conclude that $S \cong G_2(u)$, where $u$ is a power of $3$. Now Lemma~\ref{zsigm} implies immediately that $u=q$ and, therefore, $S\cong L$. This completes the proof of the lemma.

\end{proof}

\noindent{\bf Remark.} For $q=3$, the paper~\cite{Zhang_Shi_Shen} contains a mistake, and $S \in \{G_2(3), PSL_2(13)\}$ by~\cite[Table~1]{AmKhos}.

\medskip

Show that $K=1$. Consider a minimal (by order) counterexample $G$  to this claim.

\begin{lemm}\label{elemabel}
$K$ is an elementary abelian $r$-group for some $r \in \pi_1(L)$.
\end{lemm}

\begin{proof}

Let $r$ be a prime divisor of $|K|$. Since $K$ is nilpotent,
we have a decomposition $K=P \times U$, where $P$ is a Sylow $r$-subgroup of $K$ and
$U$ is a normal subgroup of $K$ such that $r\not\in\pi(U)$.
Since $U$ and $\Phi(P)$ are characteristic subgroups of $K$, the subgroup
$N=U \times \Phi(P)$ is a normal subgroup of $G$. Then $\Gamma(G/N)$ is a subgraph of $\Gamma(G)$. On the other hand, $\Gamma(L)$ is a subgraph of $\Gamma(G/N)$ and hence $\Gamma(G/N)=\Gamma(L)$. By the minimality of $G$, we infer that $N=1$ and, therefore, $K$ is an elementary abelian $r$-group.
\end{proof}


By~Lemmas~\ref{semid},~\ref{quasiF42F4} and~\ref{elemabel}, we can assume that $G=K \rtimes X$, where $K$ is an elementary abelian $r$-group for $r \in \pi_1(L)$ and $Soc(X)=S \cong L$.

\begin{lemm} $K=1$ if $L=F_4(q)$.
\end{lemm}

\begin{proof} According to~\cite[Table~5.1]{LiSaSe92}, $S$ has a subgroup $H$ isomorphic to ${^3}D_4(q)$.

Assume that $r \not \in \{2\}\cup R_{12}(q)$. By Lemma~\ref{l:r24}, $r$ is adjacent to each element from $R_{12}(q)$ in $\Gamma(KH)$. It follows from Lemma~\ref{graphF4} that $\Gamma(G) \not = \Gamma(L)$; a contradiction.

Assume that $r=2$. By Lemma~\ref{Unisingular}, $S$ is unisingular and, therefore, $2$ is adjacent to all other vertices in $\Gamma(KS)$. We arrive at a contradiction with Lemma~\ref{graphF4}.

Assume that $r \in R_{12}(q)$. Since $q$ is even, by~\cite[Table~5.1]{LiSaSe92}, $S$ has a subgroup $H_1 \cong P\Omega_8^+(q)$. Now by \cite[Table~8.50]{BrHoDo13}, $H_1$ has a subgroup $H_2 \cong P\Omega_4^-(q)\times P\Omega_4^-(q) \cong PSL_2(q^2) \times PSL_2(q^2)$. Therefore, for each $s\in R_4(q)$, a Sylow $s$-subgroup of $S$ is non-cyclic. This implies that $r$ and $s$ are adjacent in $\Gamma(KH_2)$ (see, for example, \cite[Theorem~10.3.1]{Gorenstein}).
Therefore, $r$ and $s$ are adjacent in $\Gamma(G)$; a contradiction with Lemma~\ref{graphF4}. Thus, $K=1$.
\end{proof}

\begin{lemm} $K=1$ if $L \cong G_2(q)$ for $q>3$ or $L \cong {^2}F_4(q)$ for $q>2$.
\end{lemm}

\begin{proof}
Let $p=2$ if $L \cong {^2}F_4(q)$ and $p=3$ if $L \cong G_2(q)$. By Lemma~\ref{Unisingular}, $L$ is unisingular and, therefore, if $r=p$, then $\Gamma(G)$ is connected; a contradiction.
Therefore, $r \in \pi_1(L) \setminus \{p\}$. By Lemmas~\ref{Eigenvector1} and~\ref{TiepZalThm}, $r$ is adjacent to a prime from $\pi_2(L)$; a contradiction. Thus, $K=1$.
\end{proof}

This completes the proof of Theorem~\ref{F42F4_AlmRec}.

\medskip

Since the group $L$ is known to be almost recognizable, the following natural question arises.

\begin{prob}
Suppose that $L$ is a group from the statement of Theorem~{\rm\ref{F42F4_AlmRec}}.
Find a positive integer $k$ such that $L$ is $k$-recognizable by Gruenberg-Kegel graph.
\end{prob}


\section{Unrecognizability of groups $^2B_2(q)$ and $G_2(3)$ by Gruenberg--Kegel graphs }\label{B2Section}

In this short section, we show that groups $^2B_2(q)$ with $q>2$ and $G_2(3)$ are
unrecognizable.

\begin{prop}\label{2B2} Let $G={^2}B_2(q)$, where $q>2$ is an odd power of $2$. Then $G$ is unrecognizable by Gruenberg--Kegel graph.
\end{prop}

\begin{proof} By \cite[Lemma~3.6]{GuTi03}, there exists a $4$-dimensional module $V$ over the field of order $q$ such that each nontrivial element from each maximal torus of $G$ acts fixed-point freely on $V$. Thus, $\Gamma(V \rtimes G)=\Gamma(G)$ and, therefore, $G$ is unrecognizable by \cite[Theorem~1.2]{CaMas}. \end{proof}

\begin{prop} Let $G \cong G_2(3)$. Then $G$ is unrecognizable by Gruenberg--Kegel graph.
\end{prop}

\begin{proof} Using~\cite{Atlas}, we find that $\Gamma(G_2(3))$ is the following:

\medskip

\begin{center}
	\begin{tikzpicture}
		\tikzstyle{every node}=[draw,circle,fill=black,minimum size=4pt,
		inner sep=0pt]
		
		\draw (0,0) node (2) [label=left:$2$]{}
		++ (0:1.5cm) node (3) [label=right:$3$]{}
		++ (0:1.5cm) node (7) [label=right:$7$]{}
		++ (0:1.5cm) node (13) [label=right:$13$]{}
		(2)--(3)
		;
	\end{tikzpicture}
 \end{center}
\medskip
Moreover, $\Gamma(G)=\Gamma(PSL_2(13))$. By \cite[P.~9]{AtlasBrCh} and~Lemma~\ref{BrChar}, there is a $6$-dimensional irreducible $PSL_2(13)$-module $V$ over a field of characteristic two such that all elements in
$PSL_2(13)$ of orders $7$ and $13$ act fixed-point freely on $V$. Therefore, $$\Gamma(V\rtimes PSL_2(13))=\Gamma(PSL_2(13))=\Gamma(G).$$
Thus,  by \cite[Theorem~1.2]{CaMas}, $G$ is unrecognizable by Gruenberg--Kegel graph.
\end{proof}

\section{Almost recognizability of groups $E_8(q)$ by Gruenberg--Kegel graph}\label{E8Section}
To complete the proof of Main Theorem, it remains to consider the case of groups $E_8(q)$. In~\cite{Za13}, A.~V.~Zavarnitsine proved that if $G$ is a finite group such that $\Gamma(G)=\Gamma(E_8(q))$, where $q \equiv 0, \pm 1 \pmod{5}$, then $G \cong E_8(u)$ for some $u \equiv 0, \pm 1\pmod{5}$. The aim of this section is to prove the following similar result for the remaining cases for $q$.
\begin{T}\label{E8} Let $L=E_8(q)$, where $q \equiv \pm 2 \pmod{5}$ is a prime power, and $G$ be a group such that $\Gamma(G)=\Gamma(L)$. Then $G \cong E_8(u)$ for some prime power $u$ with $u \equiv \pm 2 \pmod{5}$.
\end{T}

\begin{proof} 
Since $\Gamma(G)$ is disconnected, Lemma~\ref{vas} implies that there exists a nonabelian simple group $S$ such that $S\leq G/K\leq\operatorname{Aut}(S)$, where $K$ is the solvable radical of $G$. By the Thompson theorem on finite groups with fixed-point-free automorphisms of prime order~\cite[Theorem~1]{Thompson}, $K$ is nilpotent. By Lemmas~\ref{vas} and \ref{graphE8}, we find that $s(S) \ge s(L)=4$, $t(2,S)\geq t(2,L)=5$ and $t(S)\geq t(L)-1=11$. According to~\cite[Table~1]{ak} and \cite[Tables~2,~4, and~5]{VaVd05}, we find that  either $S \cong F_1$ or $S \cong E_8(u)$, where $u$ is a prime power.

\medskip

Suppose that $S \cong F_1$. According to \cite[Table~3]{ak}, we see that $s(S)=4$, for each $i\ge 2$, $|\pi_i(S)|=1$, and $\pi(S)\setminus \pi_1(S)=\{41, 59, 71\}$. At the same time, $\pi_2(G)=R_{15}(q)$, $\pi_3(G)=R_{24}(q)$, and $\pi_4(G)=R_{30}(q)$. By Lemma~\ref{form}, the numbers $15$, $24$, and $30$ divide $r_i-1$ for pairwise distinct primes $r_i$ from $\pi(S)\setminus \pi_1(S)$; a contradiction.

\medskip

Suppose that $S \cong E_8(u)$, where $u$ is a power of a prime $v$ and $u \equiv 0, \pm 1 \pmod {5}$.
Since $q\equiv\pm 2 \pmod{5}$, we find that $5\in R_4(q)$.
Denote $$\theta=\{r_9(q), r_{14}(q), r_7(q), r_{18}(q), r_{15}(q), r_{24}(q), r_{30}(q)\}.$$
By Lemma~\ref{graphE8}, we see that $\theta\cup\{5\}$ is a coclique of size 8 in $\Gamma(G)$.
It follows from Lemma~\ref{vas} that at least six elements of $\theta$ belong to $\pi(S)$.
Take any $r\in\pi(S)\cap\theta$. Since $r$ and 5 are nonadjacent in $\Gamma(G)$,
they are nonadjacent in $\Gamma(S)$. We know that $5\in\{v\}\cup R_1(u)\cup R_2(u)$
and hence $r\in R_{20}(u)\cup R_{15}(u)\cup R_{24}(u)\cup R_{30}(u)$ according to Lemma~\ref{graphE8}. This implies that at least two elements from $\theta\cap\pi(S)$ are adjacent in $\Gamma(S)$; a contradiction.

\medskip

Suppose that $S \cong E_8(u)$, where $v^l=u \equiv \pm 2 \pmod {5}$.
According to \cite[Table~3]{ak}, we can assume that $\pi_2(G)=R_{15}(q)$, $\pi_3(G)=R_{24}(q)$,
and $\pi_4(G)=R_{30}(q)$, while $\pi_2(S)=R_{15}(u)$, $\pi_2(S)=R_{24}(u)$,
and $\pi_4(S)=R_{30}(u)$. If $K\neq1$, then Lemma~\ref{E8nontrivK_R24} implies that
$R_{24}(u)\subset\pi_1(G)$; a contradiction since $R_{24}(u)$ must coincide with a connected component of $\Gamma(G)$ not containing 2. Therefore, we can assume that $S\leq G\leq\operatorname{Aut}(S)$.

\medskip

Now prove that $G/S=1$. If a prime $r$ divides $|G:S|$, then by \cite[Theorem~2.5.12, Definition~2.5.13, Proposition~4.9.1]{GoLySo98}, $G \setminus S$ contains a field automorphism $x$ of $S$ of order $r$ with $C_S(x) \ge E_8(u^{1/r})$. This implies that $r$ is adjacent in $\Gamma(G)$ to each prime from $\pi(E_8(u^{1/r}))$. Suppose that $r \not \in \{2, 3, 5\}$. By Lemma~\ref{calc}, $\Gamma(G)$ is connected and hence $\Gamma(G) \not = \Gamma(S)$.
If $r = 2$, then by Lemma~\ref{calc}, $R_{15}(u) \subset \pi_1(G)$; a contradiction.
If $r = 5$, then by Lemma~\ref{calc}, $r$ is adjacent in $\Gamma(G)$ to some primes from $R_{24}(u)$
and hence $R_{24}(u)\subset\pi_1(G)$; a contradiction.
Therefore, we can assume that $G/S$ is a 3-group.
By Lemma~\ref{graphE8}, we see that $r_i(q)$ is nonadjacent to 2 in $\Gamma(G)$ if and only if
$i\in\{15,20,24,30\}$. Similarly, a prime $r_i(u)$ is nonadjacent to 2 in $\Gamma(S)$ if and only if
$i\in\{15,20,24,30\}$. Since $R_{20}(q)\subset\pi_1(G)$ and $R_{20}(u)\subset\pi_1(G)$,
we infer that $R_{20}(q)=R_{20}(u)$. By Lemma~\ref{calc}, we infer that $3$ is adjacent in $\Gamma(G)$ to a prime from $R_{20}(u)$, therefore, $\Gamma(G)\not = \Gamma(S)$.

Thus, $G=S$. The proof of Theorem~\ref{E8} is complete.
\end{proof}

\medskip

\begin{cor} For each value of $q$, the group $E_8(q)$ is almost recognizable by Gruenberg--Kegel graph.
\end{cor}

\begin{proof} If $G$ is a group such that $\Gamma(G)=\Gamma(E_8(q))$, then by~\cite[Theorem~1]{Za13} and Theorem~\ref{E8}, we have $G \cong E_8(u)$ for a prime power $u$. By Lemma~\ref{NumSimple}, for a given $q$, the number of possibilities for $u$ is finite. This implies that there is only the finite number of possibilities for $G$ (up to isomorphism), in particular, $E_8(q)$ is almost recognizable by Gruenberg--Kegel graph.
\end{proof}

\medskip

\noindent{\bf Problem~2.} {\it Do there exist prime powers $q$ and $q_1$ with $q\not = q_1$ and $\Gamma(E_8(q))=\Gamma(E_8(q_1))${\rm?}}

\medskip

\begin{cor}\label{OrderGraphE8}
If $G$ is a finite group such that $\Gamma(G)=\Gamma(E_8(q))$ and $|G|=|E_8(q)|$ for some prime power $q$, then $G \cong E_8(q)$.
\end{cor}

\begin{proof} By Theorem~\ref{E8}, if $\Gamma(G)=\Gamma(E_8(q))$, then $G\cong E_8(q_1)$ for some prime power $q_1$. It is clear that a function $$f(x)=x^{120}(x^2-1)(x^8-1)(x^{12}-1)(x^{14}-1)(x^{18}-1)(x^{20}-1)(x^{24}-1)(x^{30}-1)$$ strictly monotonically increases if $x \ge 1$. Thus, if
$$f(q_1)=|G|=|E_8(q)|=f(q),$$  then $q_1=q$, and, therefore, $G \cong E_8(q)$. \end{proof}

In \cite{GrechMazVas}, it was proved that if $G$ is a simple group and $H$ is a group such that $\omega(H)=\omega(G)$ and $|H|=|G|$, then $H \cong G$. Thus, each simple group is uniquely determined by its order and spectrum. It is known that if $q$ is odd and $n\ge 3$, then $\Gamma(P\Omega_{2n+1}(q))=\Gamma(PSp_{2n}(q))$ and $|P\Omega_{2n+1}(q)|=|PSp_{2n}(q)|$ but these groups are not isomorphic. Therefore, it is natural to consider the following problem.

\medskip

\noindent{\bf Problem 3.} {\it For which simple groups $G$ is the following true: if $H$ is a group with ${\Gamma(H) =\Gamma(G)}$ and $|H| = |G|$, then $H$ is isomorphic to $G${\rm?}}

\medskip

Problem~3 was formulated by B.~Khosravi in his survey paper~\cite[Question~4.2]{BKhosravi_survey}, by A.S. Kondrat'ev in frame of the open problems session of the 13th School-Conference on Group Theory Dedicated to V.~A.~Belonogov's 85th Birthday (see~\cite[Question~4]{Maslova_Conference}), and was independently formulated by W.~Shi in a personal communication with the first author. Also Problem~3 was formulated in the paper by P.~J.~Cameron and the first author (see~\cite[Problem~2]{CaMas}). It is clear that if a simple group is quasirecognizable by Gruenberg--Kegel graph, then Problem~3 solves in the positive for this group. At the same time, Corollary~\ref{OrderGraphE8} gives a solution of Problem~3 for finite simple groups $E_8(q)$ which are not necessary quasirecognizable by Gruenberg--Kegel graph.

\section{Acknowledgements}

The first author is supported by the Ministry of Science and Higher Education of the Russian Federation, project  075-02-2022-877 for the development of the regional scientific and educational mathematical center ''Ural Mathematical Center'' (for example, Section~\ref{E8Section}). The second author is supported by the Mathematical Center in Akademgorodok under the agreement No. 075-15-2022-281 with the Ministry of Science and Higher Education of the Russian Federation (for example, Section~\ref{F4^2F4G2Section}). The third author is supported by RAS Fundamental Research Program, project FWNF-2022-0002 (for example, Section~\ref{B2Section}).

\end{document}